\documentclass[notitlepage]{amsart}

\newtheorem{theorem}{Theorem}[section]
\newtheorem{corollary}[theorem]{Corollary}
\newtheorem{lemma}[theorem]{Lemma}
\newtheorem{proposition}[theorem]{Proposition}

\theoremstyle{definition}
\newtheorem{definition}[theorem]{Definition}

\theoremstyle{remark}
\newtheorem{remark}[theorem]{Remark}
\newtheorem{example}[theorem]{Example}

\usepackage{amsthm}           	    	              	
\usepackage{amsmath}            	    	            
\usepackage{amssymb}                	    	        
\usepackage{enumerate}									
\usepackage{ulem}                                       
\usepackage{xcolor}                                     
	
\usepackage{csquotes} 		                          	

\usepackage[T1]{fontenc}                                

\usepackage[a4paper]{geometry}                          

\usepackage{tikz-cd}                                    
\usepackage[all]{xy}                                    
\usepackage{pgf,tikz}
\usepackage{mathrsfs}
\usetikzlibrary{arrows}

\newcommand{\End}{\mathrm{End}}
\newcommand{\Hom}{\mathrm{Hom}}
\newcommand{\Ext}{\mathrm{Ext}}
\newcommand{\Fac}{\mathrm{Fac}}
\newcommand{\Sub}{\mathrm{Sub}}
\newcommand{\proj}{\mathrm{proj}}
\newcommand{\T}{\mathcal{T}}
\newcommand{\F}{\mathcal{F}}
\newcommand{\C}{\mathcal{C}}
\newcommand{\B}{\mathcal{B}}

\newcommand{\X}{\mathcal{X}}

\newcommand{\U}{\mathcal{U}}

\newcommand{\f}{\mathrm{f}}
\newcommand{\modu}{\mathrm{mod}}

\newcommand{\Ker}{\mathrm{Ker}}
\newcommand{\Ima}{\mathrm{Im}}
\newcommand{\add}{\mathrm{add}}
\newcommand{\Tr}{\mathrm{Tr}}
\newcommand{\Rej}{\mathrm{Rej}}
\renewcommand{\t}{\mathrm{t}}

\renewcommand{\mod}{\mbox{mod}}

\newcommand{\rep}[1]{%
  {%
    \tiny%
    \begin{matrix}%
      #1%
    \end{matrix}%
  }%
}

\begin{document}

\title{Stratifying systems through $\tau$-tilting theory}
\author{Octavio Mendoza and Hipolito Treffinger}

\begin{abstract}
    In this paper we first show that every non-zero $\tau$-rigid $A$-module induces at least one stratifying system in the module category of $A$. 
    Moreover, we show that each of these stratifying systems can be seen as a signed $\tau$-exceptional sequence.
\end{abstract}

\maketitle


\section{Introduction}

The concept of a quasi-hereditary algebra was introduced by L. L. Scott in \cite{Scott1987} and was a key element in the characterisation of highest weight categories of finite length given by E. Cline, B. Parshall and L. L. Scott in \cite{Cline-Parshall-Scott1988}, quickly becoming a very important concept in representation theory of algebras \cite{DR2, DR1, Rin}.
One of the main features of quasi-hereditary algebras is that their homological properties are governed by the category $\F(\Delta)$ of modules filtered by a distinguished set $\Delta$ of representations known as the \textit{standard modules}.
Using the standard modules as a starting point, V. Dlab showed in \cite{Dlab} that many of the characteristics of quasi-hereditary algebras hold true for a more general class of algebras, that he called \textit{standardly stratified}.

Motivated by the homological properties of the standard modules, K. Erdmann and C. S\'aenz introduced in \cite{Erdmann2003} the notion of a stratifying system in a module category. 
They showed that one can associate to every stratifying system a module whose endomorphism algebra is standardly stratified. 
This inspired some people to start studying stratifying systems for their own sake, see for instance \cite{Marcos2006, Marcos2005a, Marcos2004, Erdmann1}.
We now recall the definition of a stratifying system as stated in \cite[Characterisation 1.6]{Marcos2004}.

\begin{definition}\label{def:stratsys}
Let $A$ be a finite dimensional $k$-algebra.
A stratifying system of size $t$ in the category $\modu\,(A)$ of finitely generated left $A$-modules is a pair $(\Theta, \leq)$ where $\Theta := \{ \Theta(i)\}_{i=1}^t$ is a family of indecomposable objects in $\modu\,(A)$ and $\leq$ is a linear order on the set $[1,t]:=\{1, \dots , t\}$ satisfying the following conditions:
\begin{itemize}
    \item[(SS1)] $\Hom_A(\Theta(j), \Theta(i))=0$ if $j > i$;
    \item[(SS2)] $\Ext^1_A(\Theta(i), \Theta(j))=0$ if $j \geq i$.
\end{itemize}
\end{definition}

The literature on stratifying systems shows that the existence of a stratifying system in the module category $\modu\,(A)$ of an algebra $A$ gives lots of information about the homological properties of $\modu\,(A)$ \cite{Erdmann1, HLM, Marcos2006, MS2006}.
However for a given algebra, little is known about the existence of a non-trivial stratifying system in its module category, with the exception of hereditary algebras \cite{Cadavid1, Cadavid2}. 

In this paper we give a constructive proof of the existence of non-trivial stratifying systems in the module category of any algebra using the $\tau$-tilting theory introduced by T. Adachi, O. Iyama and I. Reiten in \cite{AIR}. 
More precisely, we show that the construction of the standard modules $\Delta_A$ from the algebra $_AA$ as a module over itself can be generalised to every $\tau$-rigid module $M$, giving rise to a stratifying system that we denote $\Delta_M$. 
Our first main result can be summarised as follows (see Theorem \ref{thm:taustratsys},  Corollary \ref{CorApplT1} and Corollary \ref{CorApplT2}).

\begin{theorem}
Let $A$ be a finite dimensional $k$-algebra and let $M$ be basic non-zero $\tau$-rigid $A$-module. 
Then the following statements hold true.
\begin{itemize}
\item[(a)] There is at least one stratifying system $\Delta_M$ induced by $M$ whose size coincides with the number of pairwise non-isomorphic indecomposable direct summands of $M$.
\item[(b)] The stratifying system $\Delta_M$ can be completed to a stratifying system of size $n$, where $n$ is the number of isomorphism classes of simple $A$-modules.
\item[(c)] The smallest torsion class in $\modu\,(A)$ containing $\Delta_M$ is $\Fac(M)$.
\item[(d)] If $M$ has a $\Delta_M$-filtration, then $\Lambda:=\End_A(M)^{op}$ is a basic standardly stratified algebra.
Moreover, the functor $\Hom_A(M,-):\F(\Delta_M)\to \F({}_\Lambda\Delta)$ is an equivalence of categories with quasi-inverse $M\otimes_\Lambda-:\F({}_\Lambda\Delta)\to\F(\Delta_M).$
\end{itemize}
\end{theorem}

\medskip
The theory of highest weight categories and quasi-hereditary algebras inspired other important concepts in representation theory, including the notion of an exceptional sequence introduced by W. Crawley-Boevey in \cite{Crawley-Boevey1992}.
Recently, K. Igusa and G. Todorov defined in \cite{IgusaTodorov2017} the notion of signed exceptional sequences, a generalisation of exceptional sequences in the module category of a hereditary algebra, to tackle a problem arising in the theory of cluster algebras.

One of the main features of $\tau$-tilting theory is that it captures some of the combinatorial properties of cluster algebras in the module category of any finite dimensional algebra. 
Building on these ideas, A. B. Buan and R. Marsh introduced the signed $\tau$-exceptional sequences in \cite{BM} as a generalisation of signed exceptional sequences for every algebra. 

In the second part of the paper, we study the relationship between signed $\tau$-exceptional sequences and stratifying systems induced by $\tau$-rigid modules.
Our result can be summarised as follows. 
For the complete version, see Theorem \ref{thm:ss-are-taues}.  

\begin{theorem}
Let $A$ be a finite dimensional $k$-algebra and let $M$ be a non-zero $\tau$-rigid $A$-module. 
Then every stratifying system $\Delta_M$ associated to $M$ induces a signed $\tau$-exceptional sequence.
\end{theorem}

\medskip
The structure of the article is the following. 
In Section~\ref{sec:Background1}, we give the necessary background to show Theorem~1.2, and we prove it in Section~\ref{sec:Main}. 
In Section~\ref{sec:Background2}, we recall the definition of $\tau$-exceptional sequences as given in \cite{BM}.
Section~\ref{sec:comparison} is dedicated to prove Theorem~1.3.
Finally, in Section~\ref{sec:example} we study a specific example in detail. 

\medskip
{\bf Acknowledgements} We thank the anonymous referee, whose remarks have improved the readability and quality of this paper. 
The second author is grateful to Aran Tattar, Sibylle Schroll and Eduardo Marcos for the fruitful discussions. 
The authors thanks the Projects PAPIIT-Universidad Nacional Aut\'onoma de M\'exico IN103317 and IN103317. 
The second author was supported by the EPSRC funded project EP/P016294/1.

\section{Setting and background}\label{sec:Background1}
Throughout this paper, the term algebra means a non-zero finite dimensional $k$-algebra over a field $k$.
For a given algebra $A$ we denote by $\modu\,(A)$ the category of finitely generated left $A$-modules and by $\proj\,(A)$ the finitely generated projective left $A$-modules. 
The number of pairwise non-isomorphic indecomposable direct summands of $M$ is denoted by $rk(M)$ and we say that $rk(M)$ is the rank of $M$ for all $M \in \mod\,(A)$.
The Auslander-Reiten translation in $\modu\,(A)$ is denoted by $\tau$ and $D$ stands for the classical duality functor.

For a given $M\in\modu\,(A),$ we define the class $\Fac(M)$ as
$$\Fac(M):=\{X \in\modu\,(A)\; :\;\exists\,\text{ an epimorphism } M^{n} \to X \text{ for some } n\in\mathbb{N}\}.$$
Similarly, we the class $\Sub(M)$ is
$$\Sub(M):=\{X \in\modu\,(A)\; :\;\exists\,\text{ a monomorphism }  X \to M^{n} \text{ for some } n\in\mathbb{N}\}.$$
Moreover, the right perpendicular category $M^\perp$ of $M$ is 
\begin{center}
 $M^\perp:=\{X\in\modu\,(A)\;:\;\Hom_A(M, X))=0\}.$   
\end{center}
The left perpendicular category ${}^\perp M$ of $M$ is defined dually.

Let $\X$ be a full subcategory  of $\mod\,(A)$.
We say that an $A$-module $M\in\modu\,(A)$ admits an {\it $\X$-filtration} if there is a chain of submodules $0=M_0\subseteq M_1\subseteq\ldots\subseteq M_{t-1}\subseteq M_t=M$ such that each $M_i/M_{i-1}$ is isomorphic to an object in $\X$. 
The class of all $N\in\modu\,(A)$ which admits an $\X$-filtration is denoted by $\F(\X)$.
\medskip

{\sc Torsion pairs.}
A torsion pair $(\T, \F)$ in $\modu\,(A)$ is a pair $(\T, \F)$ of full subcategories of $\modu\,(A)$ such that $\T={}^\perp\F$ and $\F=\T^\perp.$
If $(\T, \F)$ is a torsion pair in $\modu\,(A),$ we say that $\T$ is a torsion class and $\F$ is a torsion free class. 
Moreover it is well-known that $\T$ is closed under quotients and extensions, while $\F$ is closed under submodules and extensions.
If $\T$ is a subcategory of $\modu\,(A)$ closed under quotients and extensions, then there exists a unique full subcategory $\mathcal{F}$ of $\modu\,(A)$ such that $(\mathcal{T}, \mathcal{F})$ is a torsion pair.
In particular, if $\mathcal{T}=\Fac(M)$ for some $M\in\modu\,(A),$ then $\mathcal{F} = M^{\perp}$.

Let $(\T, \F)$ be a torsion pair in $\mod(A)$.
Then for every $M \in \modu\,(A),$ there is a short exact sequence 
$0\to \t(M)\to M\to \f(M)\to 0,$
where $\t(M)\in\mathcal{T}$ and $\f(M)\in\mathcal{F}$, which is unique up to isomorphism.
This short exact sequence is known as the canonical short exact sequence of $M$ with respect to $(\T, \F)$.
Moreover, the natural application $\t(-): \modu\,(A) \to \modu\,(A)$ is a subfunctor of the identity functor $1_{\modu\,(A)}$  and $\f(-): \modu\,(A)  \to \modu\,(A) $ is a functor which is naturally isomorphic to $1_{\modu\,(A)}/ t.$
\medskip

{\sc Approximations.} 
Let $A$ be an algebra, $\X\subseteq\modu\,(A)$ and $M\in\modu\,(A)$.
A morphism $f:X\to M$ in $\modu\,(A)$ is a right $\X$-approximation if $X\in\X$ and $\Hom_A(X',f):\Hom_A(X',X)\to\Hom_A(X',M)$ is surjective for any $X'\in\X$.
Moreover, if $f:X\to M$ is a right $\X$-approximation and the equality $fh=f$ implies that $h$ is an automorphism, we say that $f$ is a $\X$-cover. 
The subcategory $\X$ of $\mod\,(A)$ is contravariantly finite if any $N\in\modu\,(A)$ admits a right $\X$-approximation. 
Dually, we define left $\X$-approximations and covariantly finite classes. 
It is said that $\X$ is functorially finite if it is both contravariantly and covariantly finite.
\medskip

{\sc $\tau$-tilting theory.}
In this paper we study stratifying systems via $\tau$-tilting theory, introduced by T. Adachi, O. Iyama and I. Reiten in \cite{AIR}.
We now give a brief summary of the definitions and results needed for the first part of the paper. 

\begin{definition}\cite[Definition 0.1]{AIR}
Let $A$ be an algebra.
An $A$-module $M\in\modu\,(A)$ is \textit{$\tau$-rigid} if $\Hom_A(M, \tau M)=0$. 
A $\tau$-rigid module $M$ is \textit{$\tau$-tilting} if $rk(M)=rk(A)$.
Finally, a $\tau$-rigid module $M$ is \textit{support $\tau$-tilting} if there exists an idempotent $e \in A$ such that $M$ is $\tau$-tilting as a $(A/AeA)$-module. 
\end{definition}

Let $\X$ be a full subcategory of $\modu\,(A)$ and $X \in \X$.
We say that $X\in\X$ is {\it Ext-projective} in $\X$ if $\Ext^1_A(X,Y)=0$ for all $Y \in \X$. 
Moreover, if $\X$ is a functorially finite torsion class, there are only finitely many pairwise non-isomorphic indecomposable Ext-projective modules in $\X$.
Hence the category of all the Ext-projective modules in $X$ is a full additive subcategory of $\modu\,(A)$ which is  generated by a basic 
$A$-module  denoted by $\mathbb{P}(\X).$ One of the main features of $\tau$-tilting theory is that all functorially finite torsion classes in 
$\modu\,(A)$ can be described by using support $\tau$-tilting modules, as shown in \cite{AIR, Auslander1981}.

\begin{theorem}\label{thm:fftors} Let $A$ be an algebra.
If $M\in\modu\,(A)$ is a $\tau$-rigid module, then $\Fac(M)$ is a functorially finite torsion class. 
Moreover, if $\T$ is a functorially finite torsion class in $\modu\,(A),$ then there exists a support $\tau$-tilting module $T$ such that $\T = \Fac(T)$.
\end{theorem}

An important property of $\tau$-rigid modules is that they can always be completed to a $\tau$-tilting module, as shown in \cite[Theorem 2.10]{AIR}.

\begin{theorem}\label{prop:Bongartz}
Let $A$ be an algebra and $M\in\modu\,(A)$ be a $\tau$-rigid module. 
Then ${}^\perp(\tau M)$ is a torsion class and $\B_M:=\mathbb{P}(^\perp(\tau M))$ is a $\tau$-tilting module having $M$ as a direct summand.
We say that $\B_M$ is the \textit{Bongartz completion} of $M$.
\end{theorem}

{\sc The trace and the reject.} 
Let $A$ be an algebra, $\X\subseteq\modu\,(A)$ and $M\in\modu\,(A)$. 
The trace of $\X$ in $M$ is $\Tr_\X(M):=\sum\,\{\Ima(f)\;:\;f\in\Hom_A(X,M),\,X\in\X\}$ and the reject of $\X$ in $M$ is $\Rej_\X(M):=\bigcap\,\{\Ker(f)\;:\;f\in\Hom_A(M,X),\,X\in\X\}$.
For a more detailed treatment of the trace and the reject see \cite{AF73}.
\vspace{0.2cm}

{\sc Standardly stratified algebras.} 
Given a basic algebra $A$ we have that $A$ is a left $A$-module over itself.
Then it decomposes as the direct sum $A=\bigoplus_{i=1}^n P(i),$ where $P(1),\ldots,P(n)$ is a complete list of pairwise non-isomorphic indecomposable projective $A$-modules. 
Fix some linear order $\leq$ on the set $[1,n]:=\{1,2,\ldots,n\}$ and let $\overline{P}(i):=\bigoplus_{j > i} P(j)$ for each $i \in [1,n]$.

The family of standard $A$-modules is defined as
$${}_A\Delta:=\left\{{}_A\Delta(i):=\frac{P(i)}{\mathrm{Tr}_{\overline{P}(i)}(P(i))} : i \in [1, n]\right\}.$$
It is well known \cite{DR1} that the standard modules ${}_A\Delta$ together with the linear order $\leq$ on $[1,n]$ form a stratifying system in $\modu\,(A)$ of size $n$.

\begin{definition} 
Let $A$ be a basic algebra, ${}_AA=\bigoplus_{i=1}^n P(i)$ be a decomposition of $_AA$ into indecomposable pairwise non-isomorphic projective $A$-modules and let $\leq$ be a linear order on the set $[1,n]$.
We say that the algebra $A$ is standardly stratified if ${}_AA\in\F({}_A\Delta)$. 
Moreover a stratified algebra is quasi-hereditary if $\End_A({}_A\Delta(i))$ is a division ring for all $i \in [1,n]$.
\end{definition}

{\sc Ext-projective and Ext-injective stratifying systems.}  
It was shown in \cite{Erdmann2003, Marcos2004, Marcos2005a} that given a stratifying system $(\Theta, \leq)$ there is a unique Ext-projective stratifying system $(\Theta, \underline{Q}, \leq)$ and a unique Ext-injective stratifying system $(\Theta,\underline{Y}, \leq)$ associated to $(\Theta, \leq)$, up to isomorphism. 
The definition of Ext-projective and Ext-injective stratifying systems follow.

\begin{definition}\label{def:epss}
Let $A$ be an algebra, $\Theta=\{\Theta(i)\}_{i=1}^t$ be a family of non-zero $A$-modules,
$\underline{Q}=\{Q(i)\}_{i=1}^t$ be a family of indecomposable $A$-modules
and $\leq$ be a linear order on $[1,t].$
The triple $(\Theta,\underline{Q}, \leq )$ is an Ext-projective stratifying system of size $t$ in $\modu\,(A)$
if it satisfies the following three conditions:
\begin{enumerate}
\item[(EPSS1)] $\mathrm{Hom}_A(\Theta(j),\Theta(i))=0$ for all $j> i;$
\item[(EPSS2)] for each $i\in [1,t]$there is an exact sequence 
$$0\to K(i)\rightarrow Q(i)\xrightarrow{\gamma_i}\Theta(i) \to 0$$
where  $K(i)\in\mathcal{F}(\{\Theta(j): j>i\});$
\item[(EPSS3)] $\Ext^1_A(\oplus_{i=1}^tQ(i),X)=0$ for all $X$ in $\F(\Theta)$.
\end{enumerate}
\end{definition}

\begin{definition}\label{def:eiss}
Let $A$ be an algebra,  $\Theta=\{\Theta(i)\}_{i=1}^t$ be a family of non-zero $A$-modules,
$\underline{Y}=\{Y(i)\}_{i=1}^t$ be a family of indecomposable $A$-modules
and let $\leq$ be a linear order on $[1,t]$.
The triple $(\Theta,\underline{Y}, \leq )$ is an Ext-injective stratifying system of size $t$ in $\modu\,(A)$
if it satisfies the following three conditions:
\begin{enumerate}
\item[(EISS1)] $\mathrm{Hom}_A(\Theta(j),\Theta(i))=0$ for all $j> i;$
\item[(EISS2)] for each $i\in [1,t]$ there is an exact sequence 
$$0\to \Theta(i)\xrightarrow{\lambda_i} Y(i)\to  Z(i)\to 0$$
where $Z(i)\in\mathcal{F}(\{\Theta(j): j<i\});$
\item[(EISS3)] $\Ext^1_A(X,\oplus_{i=1}^t Y(i))=0$ for all $X \in \F(\Theta)$.
\end{enumerate}
\end{definition}

Moreover, it was shown in \cite{Erdmann2003, Marcos2004} that the exact category $\F(\Theta)$ is equivalent to the category of modules filtered by the standard modules of the standardly stratified algebras $\Lambda \cong \End_A(\oplus_{i=1}^tQ(i))^{op}$ and $\Gamma \cong \End_A(\oplus_{i=1}^t Y(i))$.


\section{Stratifying systems from $\tau$-rigid modules}\label{sec:Main}

The concept of stratifying system was introduced by K. Erdmann and C. Saenz in \cite{Erdmann2003} to generalise the standard modules of standardly stratified algebras.
In this section we show that we can construct a stratifying system starting from a $\tau$-rigid module, generalising the construction of standard modules from the projective modules. 

This construction relies heavily in the indexing order of the indecomposable direct summands of $\tau$-rigid modules.
This motivates the following definition.

\begin{definition}\label{def:admorder}
Let $A$ be an algebra and $M\in\modu\,(A)$ be a basic non-zero $\tau$-rigid $A$-module.
We say that a decomposition of 
$M= \bigoplus_{i=1}^t M_i$
as the direct sum of indecomposable $A$-modules is \textit{torsion free admissible} (TF-admissible, for short) if 
$M_i \not \in \Fac\left(\bigoplus_{j>i} M_j\right),$
for every $i\in[1,t].$
\end{definition}


\begin{proposition}\label{prop:admissibledecomp}
For an algebra $A$, every basic non-zero $\tau$-rigid module in $\modu\,(A)$ admits a TF-admissible decomposition.
\end{proposition}

\begin{proof}
Let $M\in\modu\,(A)$ be a basic non-zero $\tau$-rigid $A$-module. 
We prove the above statement by induction on the number of indecomposable direct summands of $M$.

If $M$ is indecomposable, then $M=M_1$ and the claim follows immediately. 

Suppose that $M$ is the direct sum of $t\geq 2$ pairwise non-isomorphic indecomposable $A$-modules. 
Given that $M$ is a $\tau$-rigid $A$-module, Theorem \ref{thm:fftors} implies that $\Fac(M)$ is a functorially finite torsion class in $\mod\,( A)$.
Moreover, there exists a a support $\tau$-tilting module $T$ such that $\Fac(T) = \Fac(M)$ by Theorem \ref{thm:fftors}. 
Then \cite[Proposition 2.9]{AIR} implies that $M$ is a direct summand of $T$.
In other words, $T \cong M \oplus \Tilde{M}$ for some $\tau$-rigid $\Tilde{M}\in\modu\,(A)$. 
Note that $\Tilde{M}$ might be zero.

Now, since $M$ is non-zero we have that $\{0\} \subsetneq \Fac(M)$.
Then \cite[Theorem 3.1]{DIJ} implies the existence of a support $\tau$-tilting module $N\in\modu\,(A)$ which is a mutation of $T$ over an indecomposable direct summand $M_1$ and such that $\{0\} \subseteq \Fac(N) \subsetneq \Fac(M).$

We claim that $M_1$ is an indecomposable direct summand of $M$.
Suppose to the contrary that $M_1$ is a direct summand of $\Tilde{M}$. 
Then we have that $M$ is a direct summand of $N$.
This implies in particular that $\Fac(M) \subseteq \Fac(N) \subsetneq \Fac(M)$, a contradiction.
Hence $M \cong M_1 \oplus M'$, where $M'$ has $t-1$ non-isomorphic indecomposable summands.
Then, by the inductive hypothesis, we get that $M'$ admits a TF-admissible decomposition
$M' = \bigoplus_{i=2}^t M_i,$
where $M_i \not \in \Fac\left(\bigoplus_{j>i} M_j\right)$ for every $2 \leq i \leq t$.

On the other hand, note that $M'$ is also a direct summand of $N$ because we only mutate over $M_1$.
This implies in particular that $\Fac(M') \subseteq \Fac(N).$
Since $M_1 \not \in \Fac(N)$, we can conclude that $M_1 \not \in \Fac(M') = \Fac \left(\bigoplus_{j>1} M_j\right)$. 
Therefore, we get the decomposition of $M$ 
$$M \cong M_1 \oplus M' = M_1 \oplus \left( \bigoplus_{i=2}^t M_i \right) = \bigoplus_{i=1}^t M_i,$$
where $M_i \not \in \Fac\left(\bigoplus_{j>i} M_j\right)$ for every $1 \leq i \leq t$.
\end{proof}

In the following result, we show that the TF-admissible decomposition of any basic $\tau$-rigid $A$-module $M$ can be extended to a TF-admissible decomposition of its Bongartz completion $\B_M$.

\begin{proposition}\label{prop:admissiblecompletion}
Let $A$ be an algebra, $n:=rk({}_AA)$ and let $M\in\modu\,(A)$ be a basic non-zero $\tau$-rigid $A$-module with TF-admissible decomposition $M= \bigoplus_{i=1}^t M_i$. 
Then every decomposition $\B_M = \bigoplus_{j=1}^n N_j$ into indecomposable modules of the Bongartz completion of $M$ such that $M_i \cong N_{n-t+i}$  $\forall\,i\in[1,t]$ is a TF-admissible decomposition of $\B_M$.
\end{proposition}

\begin{proof}
Let $\B_M$ be the Bongartz completion of $M$.
Then $\B_M=\tilde{M}\oplus M$ by \cite[Proposition 2.9]{AIR}.

We start by proving that
$$(*)\quad X\not\in\Fac\left(M\oplus\frac{\tilde{M}}{X}\right)\;\text{ for any indecomposable }\,\,X\in\add\,(\tilde{M}).$$
Indeed, by Proposition \ref{prop:Bongartz} we have $\Fac(M\oplus\tilde{M})={}^\perp(\tau M)$.
Consider $T_X\in\modu\,(A)$ which is the mutation of $M\oplus\tilde{M}$ over an indecomposable direct summand $X$ of $\tilde{M}$. Then $\Fac(T_X)\subsetneq{}^\perp(\tau M)$ by \cite[Theorem 2.18]{AIR}.
Suppose that 
$X\in\Fac(T_X).$ Then ${}^\perp(\tau M)=\Fac(M\oplus\tilde{M})\subseteq\Fac(T_X)\subsetneq{}^\perp(\tau M),$ which is a contradiction and thus $(*)$ holds true.
\

If $i\in[1,a]$ then by $(*)$ we know that $N_i=\tilde{M_i}\not\in\Fac(\bigoplus_{j\neq i}N_j).$ 
In particular $N_i\not\in\Fac(\bigoplus_{j> i}N_j)$. 
If $i\in[a+1,n],$ then by the fact that $M= \bigoplus_{i=1}^t M_i$ is TF-admissible, we conclude that 
$N_i=M_{i-a}\not\in\Fac(\bigoplus_{k>i-a}M_k)=\Fac(\bigoplus_{j> i}N_j)$.
\end{proof}

\begin{theorem}\label{thm:taustratsys}
Let $A$ be an algebra such that $rk({}_AA)=n$ and $M\in\modu\,(A)$ be a basic non-zero $\tau$-rigid module with a TF-admissible decomposition $M= \bigoplus_{i=1}^t M_i$.
If $\f_k$ is the torsion free functor associated to the torsion pair 
$\left( \Fac(\bigoplus_{j \geq k}M_j), (\bigoplus_{j \geq k}M_j)^\perp  \right),$
then the following statements hold true.
\begin{itemize}
\item[(a)] The family $\Delta_M := \{ \Delta_M(i) := \f_{i+1}(M_i) \}_{i=1}^t$ and the natural order on $[1,t]$ form  a stratifying 
system in $\modu\,(A)$ of size $t$.
\item[(b)] There exists at least one stratifying system 
$(\Delta'_M,\leq')$ of size $n\geq t$ in $\modu\,(A),$ where $\leq'$ is the natural order on 
$[1,n],$ such that $\Delta_M(i) = \Delta'_M(n-t+i)$ for all $i \in [1,t].$
\end{itemize}
\end{theorem}

We say that $(\Delta_M, \leq)$ is the $M$-standard stratifying system of $M$ associated to the TF-admissible decomposition $M= \bigoplus_{i=1}^t M_i$.

\begin{proof} 
(a) Consider the family $\Delta_M$ as defined in the statement. 
Then, it follows directly from \cite[Lemma 4.6]{BM} that each $\Delta_M(i)$ is indecomposable since $M_i\not\in\Fac(\bigoplus_{k \geq i+1}M_k).$

We claim that $\Hom_A(\Delta_M(i), \Delta_M(j))=0$ if $i > j$.
Indeed, we have that $\Delta_M(j) \in (\bigoplus_{k > j}M_k)^\perp$ by definition.
Also $\Delta_M(i)$ is a quotient of $M_i$, implying that $\Delta_M(i)\in \Fac (\bigoplus_{k \geq i}M_k)$.
Moreover $\Fac \left(\bigoplus_{k \geq i}M_k\right)\subseteq \Fac \left(\bigoplus_{k > j}M_k\right)$.
Our claim follows since $(\Fac (\bigoplus_{k > j}M_k),(\bigoplus_{k > j}M_k)^\perp)$ is a torsion pair in $\modu\,(A)$.

To finish the proof of (a) we need to show that $\Ext^1_A(\Delta_M(i), \Delta_M(j))=0$ if $i \geq j.$ In order to do that, we assume that 
$i \geq j$ and consider the canonical short exact sequence 
$$0 \to \t_{i+1}(M_i) \to M_i \to \Delta_M(i) \to 0$$
of $M_i$ with respect to the torsion pair $\left( \Fac (\bigoplus_{k \geq i+1}M_k), (\bigoplus_{k \geq i+1}M_k)^\perp  \right).$ By
applying the functor $\Hom_A(-,\Delta_M(j))$ to the previous canonical exact sequence, we obtain the following exact sequence
$$\Hom_A(\t_{i+1}(M_i), \Delta_M(j)) \to \Ext^1_A(\Delta_M(i), \Delta_M(j)) \to \Ext^1_A(M_i, \Delta_M(j)).$$
Note that $M_i$ is Ext-projective in $\Fac(M)$ by \cite[Corollary 5.9]{Auslander1981} thus $\Ext^1_A(M_i, \Delta_M(j))=0$.
On the other hand, from $i \geq j,$ we get $\t_{i+1}(M_i)\in \Fac (\bigoplus_{k \geq i+1}M_k) \subseteq\Fac (\bigoplus_{k \geq j+1}M_k)$.
Thus $\Hom_A(\t_{i+1}(M_i), \Delta_M(j))=0$,
since $\left( \Fac (\bigoplus_{k \geq j+1}M_k), (\bigoplus_{k \geq j+1}M_k)^\perp  \right)$ is a torsion pair in $\mod(A)$ and
$\Delta_M(j) \in (\bigoplus_{k \geq j+1}M_k)^\perp$.
Then we have that $\Ext^1_A(\Delta_M(i), \Delta_M(j))=0$ as claimed. 
\

(b) It follows directly from Proposition \ref{prop:admissiblecompletion} and (a).
\end{proof}

\begin{remark}\label{rmk:notchar}
It is important to notice that the size of a stratifying system induced by a $\tau$-rigid object is bounded by the number of non-isomorphic indecomposable simple $A$-modules. 
In \cite[Remark 2.7]{Marcos2004}, the authors build a stratifying system of length 5 in the module category of an algebra of rank 4. 
This implies that Theorem \ref{thm:taustratsys} is not a characterisation of stratifying systems.
\end{remark}

Let $\X\subseteq\modu\,(A)$.
Following \cite{DIJ}, we denote by $\top(\X)$ the smallest torsion class in $\modu\,(A)$ containing 
$\X.$
We now study the minimal torsion class containing a stratifying stratifying system induced by a $\tau$-rigid module.

\begin{corollary}\label{CorApplT1} 
Let $A$ be an algebra, $M\in\modu\,(A)$ be a basic non-zero $\tau$-rigid $A$-module and let $\Delta_M$ be the $M$-standard system associated to the given TF-admissible decomposition $M= \bigoplus_{i=1}^t M_i$ of $M$. 
For each $i\in[1,t]$, consider the canonical exact sequence 
$$0\to \t_{i+1}(M_i)\to M_i\xrightarrow{\beta_i}\Delta_M(i)\to 0$$
of $M_i$ with respect to the torsion pair $(\Fac(\overline{M}_i),\overline{M}_i^\perp)$ where $\overline{M}_i = \bigoplus_{k>i} M_k$ and let $h_i:N_i\to M_i$ be the right minimal $\add(\overline{M}_i)$-approximation of $M_i$. 
Then, the following statements hold true.
\begin{itemize}
\item[(a)] $\Ima(h_i)=\t_{i+1}(M_i)=\Tr_{\overline{M}_i}(M_i)$ and  $\Delta_M(i)\simeq \frac{M_i}{\Tr_{\overline{M}_i}(M_i)}$ $\forall\,i\in[1,t].$
\item[(b)] $\beta_i:M_i\to \Delta_M(i)$ is a right minimal $\add(M)$-approximation of $\Delta_M(i)$  $\forall\,i\in[1,t].$
\item[(c)] $\F(\Delta_M)\subseteq\Fac(M)=\top(\Delta_M).$ 
\end{itemize}
\end{corollary}

\begin{proof} 
(a) We have that  $\Ima(h_i)=\t_{i+1}(M_i)$ for all $i \in [1,t]$ by \cite[Lemma 2.3]{BSTwc}.
Moreover, \cite[Proposition 8.20]{AF73} implies that $\t_{i+1}(M_i)=\Tr_{\Fac(\overline{M}_i)}(M_i)=\Tr_{\overline{M}_i}(M_i)$.

(b) Let $i\in[1,t].$ Since $M_i$ is indecomposable, it is enough to show that $\beta_i:M_i\to \Delta_M(i)$ is a right $\add(M)$-approximation of $\Delta_M(i)$.
By applying the functor $\Hom_A(X,-)$ to the exact sequence $0\to \t_{i+1}(M_i)\to M_i\xrightarrow{\beta_i} \Delta_M(i)\to 0,$ for any $X\in\add(M),$ we get the exact sequence
$$\Hom_A(X,M_i)\xrightarrow{(X,\beta_i)}\Hom_A(X,\Delta_M(i))\to\Ext^1_A(X,\t_{i+1}(M_i)).$$
 Note that $M$ is Ext-projective in $\Fac(M)$ by \cite[Corollary 5.9]{Auslander1981}, and 
 $\t_{i+1}(M_i)\in\Fac(\overline{M}_i)\subseteq\Fac(M),$ which implies that $\Ext^1_A(X,\t_{i+1}(M_i))=0.$ Thus $\beta_i:M_i\to \Delta_M(i)$ is a right $\add(M)$-approximation, proving (b).
 \
 
 (c) By using the canonical exact sequence given above, it follows that $\Delta_M\subseteq\Fac(M).$ Therefore $\F(\Delta_M)\subseteq\Fac(M)$ and $\top(\Delta_M)\subseteq\Fac(M).$ For each $i\in[1,t],$ let $M'_i:=\bigoplus_{j\neq i}\,M_j.$ Consider the torsion pair 
 $(\Fac(M'_i),{M'_i}^\perp)$ in $\modu\,(A),$  the canonical exact sequence 
 $$0\to \t'_i(M_i)\to M_i\xrightarrow{\beta'^i}\f'_i(M_i)\to 0$$
 and the $\add(M'_i)$-cover $h'_i:N_i\to M_i$ of $M_i.$ Then, by \cite[Lemma 2.3]{BSTwc}, we have that $\t'_i(M_i)=\Ima\,(h'_i).$ Hence, from \cite[Lemma 3.7]{DIJ}, it follows that $M\in\top(\bigoplus_{i=1}^t\,\f'_i(M_i));$  and thus, $\Fac(M)\subseteq \top(\bigoplus_{i=1}^t\,\f'_i(M_i)).$ On the other hand, since $\overline{M}_i\subseteq M'_i,$ we have $\f'_i(M_i)\in{M'_i}^\perp\subseteq \overline{M}_i^\perp.$ Hence there is some $\gamma_i:\Delta_M(i)\to \f'_i(M_i)$ such that $\gamma_i\beta_i=\beta'_i.$ 
 Moreover $\gamma_i$ is an epimorphism since $\beta'_i$ is so. Then, we get that  $\f'_i(M_i)\in\top(\Delta_M)$ and thus 
 $\Fac(M)\subseteq \top(\bigoplus_{i=1}^t\,\f'_i(M_i))\subseteq \top(\Delta_M).$
\end{proof}

\begin{remark} Let $A$ be basic algebra. 
Note that any decomposition into indecomposables ${}_AA=\bigoplus_{i=1}^n{}_AP(i)$ of the $\tau$-rigid module ${}_AA$ is TF-admissible. Moreover, by Corollary \ref{CorApplT1} (a), we get that the ${}_AA$-standard system coincides with the usual standard $A$-modules.
\end{remark}

In general it is not true that ${}_AA\in\F({}_A\Delta)$ unless $A$ be standardly stratified. 
Thus it is worth wondering what happens if we assume that $M\in\F(\Delta_M)$ where $M$ is a $\tau$-rigid module with TF-admissible decomposition $M= \bigoplus_{i=1}^t M_i$.

\begin{corollary}\label{CorApplT2} 
Let $A$ be an algebra, $M\in\modu\,(A)$ be a basic non-zero $\tau$-rigid $A$-module, $\Delta_M$ be the $M$-standard system associated with the TF-admissible decomposition $M= \bigoplus_{i=1}^t M_i$ of $M$ and define $\overline{M}_i:=\bigoplus_{k>i}\,M_k$.
If $M\in\F(\Delta_M),$ then  the following statements hold true.
\begin{itemize}
\item[(a)] $\Tr_{\overline{M}_i}(M_i)\in\F(\Delta_M(j)\;:\;j>i)$ for each $i\in[1,t].$
\item[(b)] $(\Delta_M,\{M_i\}_{i=1}^t,\leq)$ is an Ext-projective stratifying system in $\modu\,(A)$ of size $t,$ where $\leq$ is the natural order on $[1,t].$  
\item[(c)] $\Lambda:=\End_A(M)^{op}$ is a basic standardly stratified algebra with respect to the decomposition into indecomposables 
$\Lambda=\bigoplus_{i=1}^t {}_\Lambda P(i),$ where ${}_\Lambda P(i):=\Hom_A(M,M_i),$ and the natural order $\leq$ on $[1,t].$
\item[(d)] The functor $\Hom_A(M,-):\F(\Delta_M)\to \F({}_\Lambda\Delta)$ is an equivalence of categories with a quasi-inverse given 
by $M\otimes_\Lambda-:\F({}_\Lambda\Delta)\to\F(\Delta_M).$
\item[(e)] $\Hom_A(M,\Delta_M(i))\simeq {}_\Lambda\Delta(i)$ and  $M\otimes_\Lambda {}_\Lambda\Delta(i)\simeq \Delta_M(i)$   for each $i\in[1,t].$
\end{itemize}
\end{corollary}
\begin{proof} Let $M\in\F(\Delta_M).$ By Theorem \ref{thm:taustratsys} (a), we know that $(\Delta_M,\leq)$ is a stratifying system of size $t.$ Then, by \cite[Corollary 2.5, Proposition 2.14 (b)]{Marcos2005a} there is an Ext-projective stratifying system 
$(\Delta_M,\mathbf{Q}=\{Q(i)\}_{i=1}^t,\leq)$ and $\mathcal{P}(\Delta_M)\cap\F(\Delta_M)=\add(Q),$ where $Q:=\bigoplus_{i=1}^t\,Q(i)$ and $\mathcal{P}(\Delta_M)$ is formed by all of the $X\in\modu\,(A)$ such that $\Ext^1_A(X, \F(\Delta_M))=0.$ Since $M\in\F(\Delta_M),$ and $\Ext^1_A(M,\Fac(M))=0,$ we get from Corollary \ref{CorApplT1} (c) that $M\in\mathcal{P}(\Delta_M)\cap\F(\Delta_M)$ and so we get that $M\in\add\,(Q).$  Moreover,  \cite[Remark 2.7]{Marcos2005a} gives us  $rk(Q)=t=rk(M),$ and hence $\add(Q)=\add(M).$
\

Let $i\in[1,t].$ By Definition \ref{def:epss} and \cite[Lemma 2.3]{Marcos2005a}, there is an exact sequence 
$$0\to K(i)\to Q(i)\to \Delta_M(i)\to 0,$$ where $K(i)\in\F(\{\Delta_M(j)\;:\;j>i\})$ and $Q(i)\to \Delta_M(i)$ is an $\add(Q)$-cover of 
$\Delta_M(i).$ On the other hand, we have the canonical exact sequence 
$$0\to \t_{i+1}(M_i)\to M_i\xrightarrow{\beta_i} \Delta_M(i)\to 0$$ which is  given by the torsion pair 
$(\Fac(\overline{M}_i),\overline{M}_i^\perp).$ By Corollary \ref{CorApplT1} (b), we know that $\beta_i:M_i\to \Delta_M(i)$ is an 
$\add(M)$-cover of $\Delta_M(i).$ Using now that $\add(Q)=\add(M),$ it follows that the above two exact sequences are isomorphic. In particular we get (b), and (a) follows by Corollary \ref{CorApplT1} (a).
Once we have that $(\Delta_M,\{M_i\}_{i=1}^t,\leq)$ is an Ext-projective stratifying system in $\modu\,(A),$ the items (c), (d) and (e) follow from \cite[Theorem 3.2]{Marcos2005a}.
\end{proof}

\begin{definition}\label{tauTF} Let $A$ be an algebra.  An Ext-projective stratifying system 
$(\Theta,\{Q(i)\}_{i=1}^t,\leq)$ in $\modu\,(A),$ with the usual natural order on $[1,t],$ is said to be {\it $\tau$-torsion-free admissible} ($\tau$TF-admissible, for short) if $Q:=\bigoplus_{i=1}^t Q(i)$ is $\tau$-rigid and a TF-admissible decomposition. We denote by $\tau\mathrm{TFepss}(A)$ the class of all the Ext-projective stratifying systems which are $\tau$TF-admissible. We also consider the class 
$\mathrm{TFproper}(A)$ of all pairs $(M,\{M_i\}_{i=1}^r)$ such that $M\in\modu\,(A)$ is a non-zero basic $\tau$-rigid $A$-module and $M=\bigoplus_{i=1}^rM_i$ is a TF-admissible decomposition satisfying that $M\in\F(\Delta_M).$
\end{definition}

\begin{theorem}\label{CorApplT3} For any algebra $A,$ there are well defined functions
$$\tau\mathrm{TFepss}(A)\xrightarrow{\Upsilon}\mathrm{TFproper}(A)\xrightarrow{\Psi}\tau\mathrm{TFepss}(A),$$ 
where $\Upsilon(\Theta,\{Q(i)\}_{i=1}^t,\leq):=(\bigoplus_{i=1}^t Q(i), \{Q(i)\}_{i=1}^t)$ and 
$\Psi(M,\{M_i\}_{i=1}^r):=(\Delta_M,\{M_i\}_{i=1}^r,\leq).$ Moreover, for any $X\in\tau\mathrm{TFepss}(A)$ and $Y\in\mathrm{TFproper}(A),$ we have that $\Psi(\Upsilon( X))\simeq X$ and $\Upsilon(\Psi (Y))=Y.$
\end{theorem}
\begin{proof} Let $Y\in\mathrm{TFproper}(A).$ By Corollary \ref{CorApplT2} (b), we get that $\Psi(Y)\in\tau\mathrm{TFepss}(A).$ Moreover, it is clear that $\Upsilon(\Psi(Y))=Y.$

Let $X:=(\Theta,\{Q(i)\}_{i=1}^t,\leq)$ be an Ext-projective stratifying system in $\modu\,(A).$ By 
\cite[Remark 2.7]{Marcos2005a}, we know that all the elements of the family $\{Q(i)\}_{i=1}^t$ are pairwise non-isomorphic. Let 
$Q:=\bigoplus_{i=1}^t Q(i)$ be $\tau$-rigid and a TF-admissible decomposition. Then, we have the Ext-projective stratifying system 
$(\Delta_Q, \{Q(i)\}_{i=1}^t,\leq)=\Psi(\Upsilon( X)).$ By Definition \ref{def:epss}, \cite[Lemma 2.3]{Marcos2005a} and Corollary \ref{CorApplT1}, we have the canonical exact sequences
$$0\to K(i)\to Q(i)\xrightarrow{\gamma_i}\Theta(i)\to 0\quad\text{and}\quad 0\to \t_{i+1}(Q(i))\to Q(i)\xrightarrow{\beta_i} \Delta_Q(i)\to 0,$$ 
where $\beta_i$ and $\gamma_i$ are both $\add(Q)$-covers. Let $\Lambda:=\End_A(Q)^{op}.$  Then, by 
\cite[Theorem 3.2 (a)]{Marcos2005a} and Corollary \ref{CorApplT2} (e), it follows that
$\Theta(i)\simeq Q\otimes_\Lambda {}_\Lambda\Delta(i)\simeq{}_Q\Delta(i),$ and thus we get an isomorphism 
$\varepsilon_i:\Theta(i)\xrightarrow{\sim}\Delta_Q(i)$ for each $i\in[1,t].$ Therefore, $\varepsilon_i\gamma_i:Q(i)\to \Delta_Q(i)$ is 
an $\add(Q)$-cover,  for each $i\in[1,t].$ Hence, for each $i,$ there exists an isomorphism $\overline{\varepsilon}_i:Q(i)\to Q(i)$ such 
that $\varepsilon_i\gamma_i=\beta_i\overline{\varepsilon}_i,$ proving that $\Psi(\Upsilon( X))\simeq X$ \cite[Definition 2.4]{Marcos2005a}
\end{proof}

It is mentioned in \cite{AIR} that $\tau$-tilting can be dualised by considering $\tau^-$-rigid objects, that is objects $N \in \mod(A)$ such that 
$\Hom_A(\tau^- N, N)=0,$ where $\tau^-$ is the inverse of the Auslander-Reiten translation $\tau.$
As a consequence, all the results in this section can be dualised as well.

\section{Perpendicular categories and signed $\tau$-exceptional sequences}\label{sec:Background2}

In this second section of background, we recall the process of $\tau$-tilting reduction introduced by G. Jasso in \cite{Jasso2015} and the definition of signed $\tau$-exceptional sequences introduced by A. B. Buan and R. Marsh in \cite{BM}.
This will allow us to compare in Section \ref{sec:comparison} the stratifying systems arising from Theorem \ref{thm:taustratsys} with the signed exceptional sequences of \cite{BM}.

Let $M\in\modu\,(A)$ be a non-zero basic $\tau$-rigid module and let $\B_M\in\modu\,(A)$ be the Bongartz completion of $M$.
Following \cite{Jasso2015}, we consider the algebras 
$$B_{M}:=\text{End}_A(\B_M)^{op}\quad\text{and}\quad C_{M}:=B_{M}/\langle e_{M}\rangle,$$
where $e_{M}$ is the idempotent associated to the $B_{M}$-projective module $\Hom_{A}(\B_M, M)$.
We regard $\modu\,(C_M)$ as a full subcategory of $\modu\,(B_M)$ via the canonical embedding.
The Jasso's subcategory $J(M)$ associated with the $\tau$-rigid module $M$ is $$J(M):=M^\perp \cap\; ^\perp(\tau M).$$
We now state \cite[Theorem 3.8]{Jasso2015}.

\begin{theorem}\label{thm:tautiltred}
Let $M$ be a basic $\tau$-rigid module in $\modu\,(A)$.  
Then the functor 
$$F:=\Hom_A (\B_M, -): J(M)  \to \modu\,(C_M)$$ 
is an equivalence of categories with a quasi-inverse given by 
$$G:=\B_M\otimes_{B_M}-:\modu\,(C_M)\to J(M).$$ 
\end{theorem}

The previous result implies the existance of relative $\tau$-rigid objects in the Jasso's subcategory $J(M)$.
The following result, which appears in \cite[Proposition 3.15]{Jasso2015}, shows how to find them all. 

\begin{proposition}\label{prop:redtaurigid}
Let $M$ and $M'$ be two $A$-modules such that $M\oplus M'$ is $\tau$-rigid, $(\Fac(M), M^\perp)$ be the torsion pair associated to $M$ and let 
$$0 \to \t (M') \to M' \to \f(M') \to 0$$
be the canonical short exact sequence of $M'$ with respect to $(\Fac(M), M^\perp)$. 
Then $\f(M')$ is a $\tau$-rigid object in $J(M)$.
Moreover, every $\tau$-rigid object in $J(M)$ arises this way.
\end{proposition}

In order to define the concept of signed $\tau$-exceptional sequences we need to introduce some notation.

\begin{definition}\cite[Definitions 1.1]{BM}
Let $A$ be an algebra and consider its bounded derived category $D^b(\modu\,(A)).$ 
Define $\mathcal{C}(A)$ to be the full subcategory of $D^b(\modu\,(A))$
$$\mathcal{C}(A):= \modu\,(A) \sqcup \modu\,(A) [1] $$
whose objects are the disjoint union of the objects of the module category $\modu\,(A)$ and its shift $\modu\,( A)[1]$.
Likewise, for every subcategory $\X$ of $\modu\,(A)$, define $\mathcal{C}(\X)$ to be
$\mathcal{C}(\X):= \X \sqcup \X [1].$
An object $\mathcal{M}= M \sqcup P[1]$ in $\mathcal{C}(A)$ is said to be $\tau$-rigid if $M$ is $\tau$-rigid in $\modu\,(A),$ $P\in\proj(A)$ and $\Hom_A (P,M)=0.$ If in addition we have that $rk(M)+rk(P)=rk({}_AA),$ we say that  $\mathcal{M}$ is  support $\tau$-tilting. 
\end{definition}

\begin{remark}\label{RkJasso} 
Let $A$ be an algebra and $\mathcal{M}= M \sqcup P[1]\in\mathcal{C}(A)$ be a basic $\tau$-rigid object.  
The Jasso's subcategory for $\mathcal{M}$ is $J(\mathcal{M}):=J(M)\cap J(P)$.
Note that Theorem \ref{thm:tautiltred} can be extended to this case, and thus $J(\mathcal{M})$ is a wide subcategory of $\modu\,(A)$ which is equivalent to a module category of an algebra \cite[Theorem 4.12]{DIRRT2018}.
\end{remark}

With this notation fixed, we can now recall the definition of signed $\tau$-exceptional sequences.

\begin{definition}\cite[Definition 1.3]{BM}
Let $A$ be an algebra.  
A $t$-tuple of indecomposable objects $\left( \mathcal{U}_1, \dots, \mathcal{U}_t \right)$ in $\mathcal{C}(A)$ is a \textit{signed $\tau$-exceptional sequence} if $\mathcal{U}_t$ is a $\tau$-rigid object in $\mathcal{C}(A)$ and the tuple $\left( \mathcal{U}_1, \dots, \mathcal{U}_{t-1} \right)$ is a signed $\tau$-exceptional sequence in $\mathcal{C}\left(  J(\mathcal{U}_t)\right)$.
\end{definition}

\begin{remark} 
For a signed $\tau$-exceptional sequence $\left( \mathcal{U}_1, \dots, \mathcal{U}_t \right)$ in $\mathcal{C}(A),$ we have that $J(\U_t)$ is equivalent to the module category of an algebra $A_t,$ allowing the recursive nature of the definition.
Now, $\U_{t-1}$ is a $\tau$-rigid object inside $\C(J(\U_t))=J(\U_t)\sqcup J(\U_t)[1].$
Hence one can calculate the perpendicular category of $\U_{t-1}$ inside $\C(J(\U_t))$, and thus we get $J(\U_{t-1})$.
Proceeding inductively, we have that every signed $\tau$-exceptional sequence $\left( \mathcal{U}_1, \dots, \mathcal{U}_{t} \right)$ induces a set $\left\{J(\U_1), \dots ,J(\U_t)\right\}$ of nested wide subcategories of $\modu\,(A)$.
\end{remark}

\begin{definition}\cite[Definition 1.2]{BM}
Let $A$ be an algebra.  A $t$-tuple of indecomposable objects $\left( \mathcal{T}_1, \dots, \mathcal{T}_t \right)$ in $\mathcal{C}(A)$ is an ordered $\tau$-rigid object  if $\coprod_{i=1}^t\mathcal{T}_i$ is a $\tau$-rigid object.
\end{definition}

\begin{theorem}\cite[Theorem 5.4]{BM}\label{thm:bijection}
Let $A$ be an algebra and $n:=rk({}_AA).$ 
Then for every $t \in [1,n]$ there is a bijection between the set of ordered $\tau$-rigid objects in $\mathcal{C}(A)$ having $t$ non-isomorphic indecomposable summands and the set of signed $\tau$-exceptional sequences of length $t$.
\end{theorem}


\section{Stratifying systems and signed $\tau$-exceptional sequences}\label{sec:comparison}

In this section we show that every stratifying system that is produced by using Theorem~\ref{thm:taustratsys} can be seen as a signed $\tau$-exceptional sequence.
Moreover, we characterise all the signed $\tau$-exceptional sequences that come from such stratifying systems.
Our main result of this section is the following. 

\begin{theorem}\label{thm:ss-are-taues}
Let $A$ be an algebra and $M\in\modu\,(A)$ be a non-zero basic $\tau$-rigid $A$-module and let $\Delta_M$ be the $M$-standard system associated with the given TF-admissible decomposition $M= \bigoplus_{i=1}^t M_i$ of $M$.
Then, the $t$-tuple $\U_M:=(\U_M(1),\ldots,\U_M(t))$, where $\U_M(i):=\Delta_M(i)\sqcup 0$ $\forall\,i\in[1,t],$ is a signed $\tau$-exceptional sequence of length $t$ in $\mathcal{C}(A)$.
Moreover, every signed $\tau$-exceptional sequence $\U=(\U(1),\ldots,\U(t))$ in $\C(A)$ where $\U(i)=U(i)\sqcup 0$ and $U(i)\in\modu\,(A)$ $\forall\,i\in[1,t]$ arises this way.
\end{theorem}

The main difference between the two constructions is that the signed $\tau$-exceptional sequences in $\C(A)$ are constructed recursively, while the construction of the stratifying system $\Delta_M$ is direct.
Therefore, the first step in order to show the compatibility between both constructions, we need to show that the perpendicular categories that we find are the same.

\begin{lemma}\label{lem:redtors}
Let $A$ be an algebra, $M\oplus M'\in\modu\,(A)$ be a basic $\tau$-rigid module and let $\f : \modu\,(A) \to \modu\,(A)$ be the torsion free functor associated to the torsion pair $(\Fac( M), M^\perp).$
Then, $\Fac (M\oplus M') \cap M^\perp  = \Fac (\f( M')) \subseteq J(M)$.
\end{lemma}

\begin{proof} 
By \cite[Proposition 3.15]{Jasso2015} it follows that $\f (M\oplus M')$ is Ext-projective in $\Fac (M\oplus M') \cap M^\perp$.
In particular $\f(M') \in \Fac (M \oplus M')\cap M^\perp$ and thus $\Fac (\f (M')) \subseteq \Fac (M \oplus M')\cap M^\perp$ because is a torsion class in $J(M)$ by \cite[Theorem 3.12]{Jasso2015}.
\

Take $N \in \Fac(M \oplus M') \cap M^{\perp}$. Since $\Fac(M \oplus M') \cap M^{\perp} \subseteq \Fac(M'),$ we have that $N$ is in $\Fac(M')$.
Then there exists an epimorhism $p: M'' \to N$ with $M''$ in $\add( M').$
Now, consider the canonical short exact sequence 
$$0 \to \t(M'') \to M'' \to \f(M'') \to 0$$
of $M''$ with respect to the torsion pair $(\Fac(M), M^{\perp})$.
By applying the functor $\Hom(-,N)$ to the previous short exact sequence, we obtain the exact sequence
$$0 \to \Hom_A(\f (M''), N) \to \Hom_A(M'', N) \to \Hom_A(\t (M''), N).$$
Now, since $N \in \Fac(M \oplus M') \cap M^{\perp}$ we have that $N \in M^\perp$.
Thus $\Hom_A(\t (M''), N)=0$.
Hence every map from $f\in \Hom_A(M'',N)$ factors through a map $f' \in \Hom_A(\f(M''), N).$
In particular, the epimorphism $p: M'' \to N$ factors through a map $p' : \f(M'') \to N$ which is necessarily an epimorphism.
Therefore $N \in \Fac(\f(M'))$ and the result follows.
\end{proof}

\begin{corollary}\label{cor:pasoinductivo}
Let $A$ be an algebra, $M\in\modu\,(A)$ be a non-zero basic $\tau$-rigid $A$-module, $M=\bigoplus_{i=1}^{t} M_i$ be a TF-admissible decomposition of $M,$ and let $\f_t:\modu\,(A)\to\modu\,(A)$ be the torsion free functor associated to the torsion pair $(\Fac(M_t) , M_t^\perp)$.  
Then $\f_t(M)$ is a $\tau$-rigid object in $J(M_t)$ and $ \f_t (M)=\bigoplus_{i=1}^{t-1} \f_t (M_i)$ is a TF-admissible decomposition of $\f_t(M)$.
\end{corollary}

\begin{proof}
The fact that $\f_t(M)$ is a $\tau$-rigid object  in $J(M_t)$ follows from Proposition \ref{prop:redtaurigid}.
Similarly, by Proposition \ref{prop:redtaurigid}, we get that $\f_t (M_i)$ is a $\tau$-rigid object in $J(M_t),$ for all $1 \leq i \leq t-1.$ 
Moreover, for every $1\leq i \leq t-1$, the object $\f_t(M_i)$ is indecomposable by \cite[Lemma 4.6]{BM}.
Note that these objects are pairwise non-isomorphic by \cite[Lemma 4.7]{BM}.
Therefore we have that 
$$ \f_t (M)=  \bigoplus_{i=1}^{t-1} \f_t (M_i)$$ 
is a decomposition of $\f_t(M)$ since $\f_t (M_t) = 0.$
Finally, the fact that the previous decomposition is TF-admissible follows from Lemma \ref{lem:redtors}.
\end{proof}

\begin{lemma}\label{lem:comptorsfreefunct}
Let $A$ be an algebra, $M\oplus M'\in\modu\,(A)$ be a $\tau$-rigid module and let $\f_M$ be the torsion free functor associated to the torsion pair $(\Fac(M), M^\perp)$.
Then, the canonical short exact sequences of $X$ with respect to the torsion pairs $(\Fac(M\oplus M'), (M\oplus M')^\perp)$ in $\modu(A)$ and $(\Fac(\f_M(M')), \f_M(M')^{\perp})$ in $J(M)$ are isomorphic for all $X\in J(M)$.
\end{lemma}

\begin{proof}
Let $X\in J(M)$ and let $\t_{M\oplus M'}(X)$ and $\tilde{\t}_{\f(M)}(X)$ be the torsion submodules of $X$ with respect to the torsion pairs 
$(\Fac(M\oplus M'), (M\oplus M')^\perp)$ in $\mod (A)$ and $(\Fac(\f_M(M')), \f_M(M')^{\perp})$ in $J(M)$, respectively.
Note that it is enough to show that $\t_{M\oplus M'}(X)$ is isomorphic to $\tilde{\t}_{\f(M)}(X)$.

By Lemma~\ref{lem:redtors} we have that $\Fac (\f( M')) = \Fac (M\oplus M') \cap M^\perp$. 
Thus $\tilde{\t}_{\f(M)}(X) \in \Fac (M\oplus M') \cap M^\perp$.
In particular, $\tilde{\t}_{\f(M)}(X) \in \Fac (M\oplus M')$. 
Then there exist a monomorphism $f : \tilde{\t}_{\f(M)}(X) \to \t_{M\oplus M'}(X)$ because $t_{M\oplus M'}(X)$ contains all submodules of $X$ that are isomorphic to a module in $\Fac(M \oplus M')$.

On the other hand, we have that $X \in J(M) = M^\perp \cap {}^\perp\tau M$.
Then $X$ is in the torsion free class $M^\perp$.
In particular, this implies that $\t_{M\oplus M'}(X) \in M^\perp$ since every torsion free class is closed under submodules.
Also, we have that $\t_{M\oplus M'}(X) \in \Fac(M\oplus M')$ by definition. 
Then $\t_{M\oplus M'}(X) \in \Fac(M\oplus M') \cap  M^\perp = \Fac (\f( M'))$ by Lemma~\ref{lem:redtors}.
Hence, there exists a monomorphism $f' : \t_{M\oplus M'}(X) \to \tilde{\t}_{\f(M)}(X)$ by the properties of the torsion functor. 
Thus $\t_{M\oplus M'}(X) \cong \tilde{\t}_{\f(M)}(X)$ for all $X \in J(M)$.
\end{proof}

\begin{proof}[Proof of Theorem \ref{thm:ss-are-taues}]
Let $M\in\modu\,(A)$ be a non-zero, basic and $\tau$-rigid $A$-module with a TF-admissible decomposition $M=\bigoplus_{i=1}^t M_i.$ 
We first show, by induction, that 
$$\U_M:=(\U_M(1),\ldots,\U_M(t)),$$ 
where $\U_M(i):=\Delta_M(i)\sqcup 0$ $\forall\,i\in[1,t],$ is a signed $\tau$-exceptional sequence. 

First, by construction of $\Delta_M$, we have that $\Delta_M(t)=M_t$.
Then $\U_t:=\Delta_M(t)\sqcup0$ is a support $\tau$-rigid object in $\C(A)$ because $\Delta_M(t)$ is a $\tau$-rigid indecomposable $A$-module. 

Now, Corollary \ref{cor:pasoinductivo} implies that 
$$\f_t (M )=\bigoplus_{i=1}^{t-1} \f_t (M_i)$$ 
is a TF-admissible decomposition of $\f_t (M)$, where $\f_t$ is the torsion free functor associated to the torsion pair $\left(\Fac (M_t), M_t^\perp \right)$. 
Then we have that $\Delta_M(t-1)=f_t(M_{t-1})$ is an indecomposable $\tau$-rigid module in $J(M_t)$ by \cite[Lemma 4.6]{BM} and \cite[Lemma 4.7]{BM}, implying that
$\U_{t-1}:= \Delta_M(t-1) \sqcup 0$ is an indecomposable support $\tau$-rigid in $\C(\U_t)$.
Hence $\{\U_{t-1}, \U_t\}$ is a signed $\tau$-exceptional sequence of length two.

Let $\tilde{\f}_{t-1}$ be the torsion free functor associated to the torsion pair induced by $\f_t (M_{t-1})$ within $J(\U_t)= M_t^\perp \cap \; ^\perp(\tau M_t)$ and $\f_{t-1}$ be the torsion free functor associated to the torsion pair 
$$\left(\Fac (\oplus_{i=t-1}^t M_i) ,  (\oplus_{i=t-1}^t M_i)^\perp \right)$$ 
generated by $M_{t-1}\oplus M_t$. 
By Lemma~\ref{lem:comptorsfreefunct} we have that $\tilde{\f}_{t-1}(\f_t (M_{i})) \cong \f_{t-1}(M_{i})$ for all $1 \leq i \leq t-2$.
Thus, again by \cite[Lemma 4.6]{BM} and \cite[Lemma 4.7]{BM},  we have that $\Delta_M(t-2) = f_{t-1}(M_{t-2})$ is an indecomposable $\tau$-rigid object in $J(\U_{t-1})$, implying that
$$\left\{ \U_{t-2}:= \Delta_M(t-2) \sqcup 0 \;, \; \U_{t-1} := \Delta_M(t-1)\sqcup 0 \; ,\; \U_t:= \Delta_M(t) \sqcup 0 \right\}$$
is a signed $\tau$-exceptional sequence of length three. 

Now we do the inductive step.
Suppose that
$$\left\{ \U_{k}:= \Delta_M(k) \sqcup 0 \;, \dots \; , \; \U_{t-1} := \Delta_M(t-1)\sqcup 0 \; ,\; \U_t:= \Delta_M(t) \sqcup 0 \right\}$$
is a signed $\tau$-exceptional sequence of length $k$.
Once again Lemma~\ref{lem:comptorsfreefunct} implies that 
$\tilde{\f}_{k}(\f_{k-1} (M_{i})) \cong \f_{k}(M_{i})$
for all $1 \leq i \leq k-1$, where $\tilde{\f}_{k}$ is the torsion free functor associated to the torsion pair induced by $\f_{k-1}(M_{k})$ within $J(\U_{k-1}).$
Then we can apply one more time \cite[Lemma 4.6]{BM} and \cite[Lemma 4.7]{BM} to conclude that 
$$\left\{ \U_{k+1}:= \Delta_M(k+1) \sqcup 0 \;, \;\U_{k}:= \Delta_M(k) \sqcup 0 \;,\; \dots \; ,\; \U_t:= \Delta_M(t) \sqcup 0 \right\}$$
is a signed $\tau$-exceptional sequence of length $k+1$.
\medskip

For the moreover part, let 
$$\left\{ \U_i = N_i \sqcup 0 : N_i \in \mod\,(A)  \text{ for all $1\leq i\leq t$}\right\}$$
be a signed $\tau$-exceptional sequence as in the statement. 
Then, by Theorem~\ref{thm:bijection} there exists an ordered $\tau$-rigid object $\mathcal{M}= \bigoplus_{i=1}^t \mathcal{M}_i$ in $\C_A$ inducing this signed $\tau$-exceptional sequence.
By hypothesis, $\U_i = N_i \sqcup 0$ where $N_i \in \mod\,(A)$ for all $1\leq i\leq t$.
Hence \cite[Remark~5.12]{BM} implies that 
$$N_i = \tilde{\f}_t(\tilde{\f}_{t-1}(\dots (\tilde{\f}_{i+1}(\mathcal{M}_i)))$$
for all $1 \leq i \leq t$, where $\tilde{\f}_k$ is the torsion free functor associated to the torsion pair induced by $\U_k$ in $J(\U_{k+1})$.
This implies that $\mathcal{M}_i=M_i\sqcup 0$ where $M_i \in \modu\,(A)$ for all $1 \leq i \leq t$.
Moreover, by Lemma~\ref{lem:comptorsfreefunct} we have that $N_i \cong \f_{i+1}(M_i)$ for all $1\leq i \leq t$, where $\f_{i+i}$ is the torsion free functor associated to the torsion pair $\left(\Fac (\oplus_{k=i+1}^t M_k) ,  (\oplus_{k=i+1}^t M_k)^\perp \right)$.
Since $N_i$ is non-zero for all $1 \leq i \leq t$, we have that $M_i \not \in \Fac\left(\bigoplus_{j>i} M_j\right)$.
Then $M=\bigoplus_{i=1}^t M_i$ is a $\tau$-rigid module in $\mod\,(A)$ and this is a TF-admissible decomposition of $M$. 
Hence Theorem~\ref{thm:taustratsys} implies that $\Delta_M = \{\Delta_M(i) = \f_{i+1}(M_i) : 1 \leq i \leq n\}$ is a stratifying system and $N_i \cong \Delta_M(i)$ for all $i.$
\end{proof}


\section{Example}\label{sec:example}

We finish the paper by studying the number of stratifying systems induced by the $\tau$-tilting modules of an algebra. 

\begin{example}
Let $A$ be the quotient path $k$-algebra given by the quiver 
$$
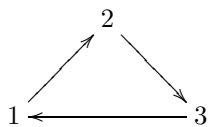
\xymatrix{
  & 2\ar[dr]& \\
  1\ar[ru] & & 3\ar[ll] }$$
and the third power of the ideal generated by all the arrows.
The Auslander-Reiten quiver of $A$ can be seen in Figure \ref{fig:Ar-quiverA}.
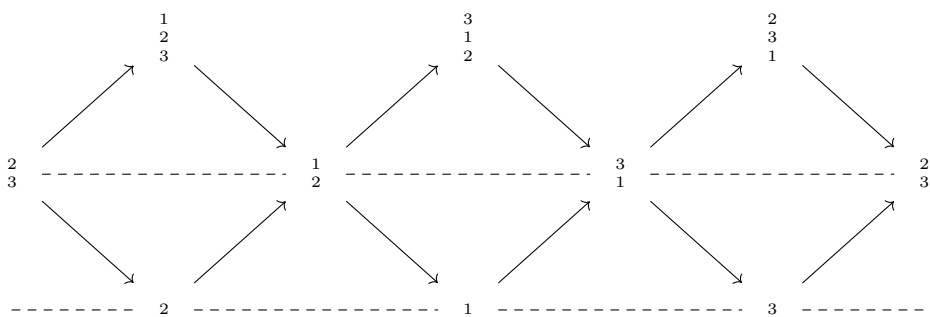
\begin{figure}[h]
    \centering
			\begin{tikzpicture}[line cap=round,line join=round ,x=2.0cm,y=1.8cm]
				\clip(-2.2,-0.1) rectangle (4.1,2.5);
					\draw [->] (-0.8,0.2) -- (-0.2,0.8);
					\draw [->] (1.2,0.2) -- (1.8,0.8);
					\draw [->] (3.2,0.2) -- (3.8,0.8);
					\draw [->] (-1.8,1.2) -- (-1.2,1.8);
					\draw [->] (0.2,1.2) -- (0.8,1.8);
					\draw [->] (2.2,1.2) -- (2.8,1.8);
					\draw [dashed] (-0.8,0.0) -- (0.8,0.0);
					\draw [dashed] (1.2,0.0) -- (2.8,0.0);
					\draw [dashed] (-1.8,1.0) -- (-0.2,1.0);
					\draw [dashed] (0.2,1.0) -- (1.8,1.0);
					\draw [dashed] (2.2,1.0) -- (3.8,1.0);
					\draw [dashed] (-2.0,0.0) -- (-1.2,0.0);
					\draw [dashed] (3.2,0.0) -- (4.0,0.0);
					\draw [->] (0.2,0.8) -- (0.8,0.2);
					\draw [->] (2.2,0.8) -- (2.8,0.2);
					\draw [->] (-1.8,0.8) -- (-1.2,0.2);
					\draw [->] (-0.8,1.8) -- (-0.2,1.2);
					\draw [->] (1.2,1.8) -- (1.8,1.2);
					\draw [->] (3.2,1.8) -- (3.8,1.2);
				
				\begin{scriptsize}
					\draw[color=black] (-1,0) node {$\rep{2}$};
					\draw[color=black] (1,0) node {$\rep{1}$};
					\draw[color=black] (3,0) node {$\rep{3}$};
					\draw[color=black] (-2,1) node {$\rep{2\\3}$};
					\draw[color=black] (0,1) node {$\rep{1\\2}$};
					\draw[color=black] (2,1) node {$\rep{3\\1}$};
					\draw[color=black] (4,1) node {$\rep{2\\3}$};
					\draw[color=black] (-1,2) node {$\rep{1\\2\\3}$};
					\draw[color=black] (1,2) node {$\rep{3\\1\\2}$};
					\draw[color=black] (3,2) node {$\rep{2\\3\\1}$};
				\end{scriptsize}
			\end{tikzpicture}
\caption{The Auslander-Reiten quiver of $A$}
    \label{fig:Ar-quiverA}
\end{figure}
Note that every module is represented by its Loewy series and both copies of $\rep{2\\3}$ should be identified, so the Auslander-Reiten quiver of $A$ has the shape of a cylinder. 
One can see that there are 10 $\tau$-tilting modules in $\modu\,(A)$.
In the table \ref{table:chambers},  we give a complete list of them and we indicate how many stratifying systems and $\tau$TF-admissible Ext-projective stratifying systems they induce, up to isomorphism. 
In particular, from the first line of the table, we get that this algebra is not standardly stratified under any linear order on the set $\{1,2,3\}.$  

\begin{table}[h]
\centering
\begin{tabular}{|c|c|c|}\hline
 $\tau$-tilting module & ind. strat. syst.  & $\tau$TFepss up to isom.\\\hline
 $\rep{1\\2\\3} \oplus \rep{2\\3\\1} \oplus \rep{3\\1\\2}$ & 6 & 0 \\\hline
 $\rep{1\\2\\3} \oplus \rep{2\\3\\1} \oplus \rep{2}$ & 2 & 0 \\\hline
 $\rep{1\\2\\3} \oplus \rep{3\\1\\2} \oplus \rep{1}$ & 2 & 0\\\hline
 $\rep{2\\3\\1} \oplus \rep{3\\1\\2} \oplus \rep{3}$ & 2 & 0 \\\hline
 $\rep{1\\2\\3} \oplus \rep{1\\2} \oplus \rep{2}$ & 3  & 3\\\hline
 $\rep{1\\2\\3} \oplus \rep{1\\2} \oplus \rep{1}$ & 1 & 1 \\\hline
 $\rep{2\\3\\1} \oplus \rep{2\\3} \oplus \rep{3}$ & 3 & 3\\\hline
 $\rep{2\\3\\1} \oplus \rep{2\\3} \oplus \rep{2}$ & 1 & 1\\\hline
 $\rep{3\\1\\2} \oplus \rep{3\\1} \oplus \rep{1}$ & 3 & 3\\\hline
 $\rep{3\\1\\2} \oplus \rep{3\\1} \oplus \rep{3}$ & 1 &1 \\\hline
 
\end{tabular}
\vspace{0.2cm}
\caption{The number of stratifying systems and $\tau$TF-admissible Ext-projective stratifying systems (up to isomorphism) induced by each $\tau$-tilting module in $\modu(A).$}
\label{table:chambers}
\end{table}

\end{example}

\footnotesize

\vskip3mm \noindent Octavio Mendoza Hern\'andez:\\ Instituto de Matem\'aticas, Universidad Nacional Aut\'onoma de M\'exico\\
Circuito Exterior, Ciudad Universitaria,
C.P. 04510, M\'exico, D.F. MEXICO.\\ {\tt omendoza@matem.unam.mx}

\vskip3mm \noindent Hipolito Treffinger:\\ Department of Mathematics,  University of Leicester\\
University Road (LE1 7RH)\\
Leicester, United Kingdom.\\ {\tt hjtc1@leicester.ac.uk}
 	
\end{document}